\documentclass{article}
\usepackage{scrextend}
\usepackage{amssymb}
\usepackage{amsthm}
\usepackage{amsmath}
\newtheorem{theorem}{Theorem}
\newtheorem{defi}{Definition}
\newtheorem{lemma}{Lemma}
\newtheorem{cor}{Corollary}
\newtheorem{prop}{Proposition}
\begin{document}
\title{\bf{Eigenspaces of Newforms with Nontrivial Character}}
\author{M. Karameris}
\date{}
\maketitle
\begin{abstract}
  Let $\mathcal{S}_{k}(\Gamma_0(N),\chi)$ denote the space of holomorphic cuspforms with Dirichlet character $\chi$ and modular subgroup $\Gamma_0(N)$. We will characterize the space of newforms $\mathcal{S}_{k}^{new}(\Gamma_0(N),\chi)$ as the intersection of eigenspaces of a particular family of Hecke operators, generalizing the work of Baruch-Purkait \cite{BarSom} to forms with non-trivial character. We achieve this by obtaining representation theoretic results in the $p$-adic case which we then de-adelize into relations of classical Hecke operators. 
  \end{abstract}
\section*{Introduction}
The aim of this work is to provide a link between the theory of new vectors described by Casselman and the classical theory of Atkin and Lehner for the space of cuspforms $\mathcal{S}_{k}(\Gamma_0(N),\chi)$.
 In \cite{BarSom1}, Baruch and Purkait worked with Loke and Savin's method \cite{LokSav} to characterize the space of integral weight newforms $S^{new}_{2k}(\Gamma_0(N))$ as the intersection of eigenspaces of Hecke operators that appear from adelizing a $p-$adic operator in the corresponding Hecke algebra. In particular for $p|N$ they found two families of Hecke operators and their conjugates under involution, which they denoted $Q_p,Q_p'$ and $S_{p^n,n-1},S_{p^n,n-1}'$ respectively, for which they obtained that:
\begin{theorem}[\textbf{Baruch-Purkait}]
    Consider the space of level $N=\prod\limits_{j=1}^np_j\prod\limits_{i=1}^m{q_i}^{a_i}$ cuspforms $\mathcal{S}_{k}(\Gamma_0(N))$. Then the space of newforms $\mathcal{S}_{k}^{new}(\Gamma_0(N))$ is the intersection of the $-1$ eigenspaces of the Hecke operators $Q_{p_j}, Q'_{p_j}$ and the $0$ eigenspaces of $S_{{q_i}^{a_i},a_i-1},S'_{{q_i}^{a_i},a_i-1}$ for all indices $i,j$ in the above product. 
\end{theorem}
In the process, Baruch and Purkait described the structure of a Hecke algebra of compactly supported, bi-invariant functions on $GL_2(\mathbb{Z}_p)$ with respect to an open compact subgroup $K_0(p^n)$ defined below. In particular they use this algebra to determine the finite dimensional representations of $K$ containg $K_0(p^n)$-fixed vectors and explicitly describe these vectors and the action of the Hecke algebra on them. This algebra contains a unique $K_0(p^n)$-fixed vector in the sense of Casselman \cite{CassLehn} and there is an operator whose $0$ eigenspace contains only this eigenvector. Adelizing this operator, Baruch and Purkait find a new Hecke operator not previously known in the classical Atkin-Lehner theory.
\par 
We generalize the results in \cite{BarSom} to the case of non-trivial character forms $\mathcal{S}_k(\Gamma_0(N),\chi)$. We use a similar approach but work with Hecke algebras that are $\chi$ bi-invariant (see subsection \ref{TwistHecke}) to obtain families of classical Hecke operators whose common eigenspaces characterize the space $\mathcal{S}^{new}_k(\Gamma_0(N),\chi)$. The operators Baruch and Purkait used are special cases of these Hecke operators for $\chi\equiv 1$. 
\par 
In section \ref{MainRes} we introduce the notation and state the main result. In section \ref{HeckeAlg} we first recall some known results on the Iwahori Hecke algebra and then exhibit the structure of the twisted Hecke subalgebras for higher level Iwahori subgroups of $GL_n(\mathbb{Z}_p)$. We use this structure to describe the irreducible representations of $K$ with $\chi$-equivariant vectors and determine the action of the Hecke algebra on these vectors. Section \ref{AdeltoCl} is where we de-adelize the $p$-adic operators we obtained from the previous section into classical operators. In section \ref{Sqfree} we use these operators to prove the main theorem for square free level forms and in section \ref{GenLev} we obtain the main theorem.
 \vspace{1cm}
 \section*{Acknowledgements}
My heartfelt thanks go to my advisor, Moshe Baruch, for not only introducing me to this problem but also for our extensive discussions that made this work possible.
 \newpage 
\section{Main Result} \label{MainRes}
For $N\in\mathbb{N}$, let $\chi$ be a Dirichlet character $\mod N$. Let $N=p_1^{n_1}\dots p_m^{n_m}$ be the prime decomposition of $N$. Then $\chi$ can be decomposed into a product of Dirichlet characters $\chi=\chi^{(p_1^{n_1})}\dots\chi^{(p_m^{n_m})}$ (see \cite[p. 44]{Gold}). Throughout the paper we assume that either $\chi(-1)=-1$ or that the weight $k$ is even so that $\mathcal{S}_{k}(\Gamma_0(N),\chi)$ is non trivial.
For any matrix $A=\begin{pmatrix}
    a & b \\
    c & d
\end{pmatrix}\in GL_2(\mathbb{R})$ with positive determinant and $f$ a complex valued function define the slash operator: 
\[f|A(z)=(ad-bc)^{\frac{k}{2}}(cz+d)^{-k}f(\frac{az+b}{cz+d}).\]
For any matrix $A=\begin{pmatrix}
    a & b \\
    c & d
\end{pmatrix}\in \Gamma_0(N)$ the functions $f\in\mathcal{S}_{k}(\Gamma_0(N),\chi)$ satisfy the following defining relation: 
\[f|A(z)=(ad-bc)^{\frac{k}{2}}(cz+d)^{-k}f(\frac{az+b}{cz+d})=\chi(d)f(z).\]
For a prime $p$ such that $p^n\parallel N$ (i.e. $p^n|N$ but $p^{n+1}\nmid N$) and $N=p^nM$ consider the well known Atkin-Lehner-Li operators: 
\[\text{the Atkin-Lehner involution: }W_{p^n}(f)(z)=f|\begin{pmatrix}
    p^n\beta & 1 \\
    N\gamma & p^n
\end{pmatrix}(z) \]  with $\beta\in\mathbb{Z}$ such that $p^{2n}\beta-N\gamma=p^n$, 
\[\text{the operator }U_p(f)(z)=p^{\frac{k}{2}-1}\sum_{s=0}^{p-1}f(\frac{z+s}{p}),\] and its normalization as it appears in \cite[§0.]{Li-Atkin}:
\[\tilde{U}_p(f)(z)=p^{1-k}U_p.\] 
The operators $U_p$ play a central part in the study of modular forms since they provide the Fourier series coefficients of $f$ in the case $p|N$ (see \cite[Theorem 3]{Atkin-Lehner}).
We also define the new operators, a trivial character version of which appears in \cite{BarSom} and also in the proof of \cite[Theorem 2.1]{Li-Atkin}:
\[Q_p=\bar{\chi}^{(M)}(p)\tilde{U}_pW_p= \bar{\chi}^{(M)}(p)p^{1-\frac{k}{2}}U_pW_p, \text{ for } p\parallel N\]
\[\text{and its conjugate }Q_p'=W_pQ_pW_p^{-1},\text{ for } p\parallel N\] as well as the twist of an operator first considered by Baruch and Purkait \cite{BarSom}:
\[S_{p^n,n-1}(f)(z)=f+\sum_{s\in(\mathbb{Z}/p\mathbb{Z})^*}\overline{\chi(A_s)}f|A_s(z),\] where $A_s=\begin{pmatrix}
    a_s & b_s \\
    N/p & p-sN/p^n
\end{pmatrix}$
with $a_s,b_s$ integers such that $det(A_s)=1$, 
\[\text{and its conjugate }S_{p^n,n-1}'=W_{p^n}S_{p^n,n-1}W_{p^n}^{-1}.\] 
Using these operators we obtain the following result:
\begin{theorem} \label{mainth}
    Consider the space $\mathcal{S}_{k}(\Gamma_0(N),\chi)$ where $N=\prod\limits_{j=1}^{m_1}p_j\prod\limits_{i=1}^{m_2}{q_i}^{n_i}$ with $\prod\limits_{j=1}^{m_1}p_j$ being the squarefree part of $N$ and write $\chi=\prod\limits_{j=1}^{m_1}\chi^{(p_j)}\prod\limits_{i=1}^{m_2}\chi^{({q_i}^{n_i})}$. Then $\mathcal{S}_{k}^{new}(\Gamma_0(N),\chi)$ is the intersection of the $-1$ eigenspaces of the Hecke operators $Q_{p_j}, Q'_{p_j}$ and the $0$ eigenspaces of the operators $S_{{q_i}^{n_i},n_i-1},S'_{{q_i}^{n_i},n_i-1}$ for every $i,j$ in the above product for which $\chi^{(p_j)}\equiv 1$ or $\chi^{({q_i}^{n_i})}$ is imprimitive. 
\end{theorem}
\section{Twisted p-adic Hecke algebras and representations of $GL_2(\mathbb{Z}_p)$} \label{HeckeAlg}
Let $\chi_p:\mathbb{Q}_p^{\times}\to\mathbb{C}^{\times}$ be a $p-$adic Hecke character. In this section we will explicitly describe the structure of the Hecke algebras of $\chi_p$-bi-invariant functions on $K=GL_2(\mathbb{Z}_p)$. We will use these results to obtain a classification of the smooth irreducible finite dimensional representations of $K$ which have $\chi_p$ equivariant vectors. 
\subsection{The Iwahori Hecke Algebra}
Let $G=GL_2(\mathbb{Q}_p)$ and consider the family of subgroups $K_0(p^n)=\{\begin{pmatrix}
    a & b \\
    c & d
\end{pmatrix}\in K:c\in p^n\mathbb{Z}_p\}$. Let $\chi_p$ be a $p-$adic character, $\chi_p:\mathbb{Q}_p^{\times}\to\mathbb{C}^{\times}$ of conductor $1+p^r\mathbb{Z}_p$ with $r\leq n$. Abusing notation we extend $\chi_p:K_0(p^n)\to\mathbb{C}^{\times}$ as $\chi_p(\begin{pmatrix}
    a & b \\
    c & d
\end{pmatrix})=\chi_p(d)$.
We first consider the case of the Iwahori Hecke algebra $$\mathcal{H}(G/\!\!/K_0(p),\chi_p)=\{f:G\to\mathbb{C}|f(kxk')=\chi_p(kk')f(x),\forall x\in G, k,k'\in K_0(p)\}$$ where $\chi_p$ is a $p-$adic character of conductor at most $1+p\mathbb{Z}_p$. It is obvious that if $\chi_p\equiv 1$ then we get the full Iwahori Hecke algebra $\mathcal{H}(G/\!\!/K_0(p))$. We define the elements $d(p^n)=\begin{pmatrix}
    p^n & 0 \\
    0 & 1
\end{pmatrix}, w(p^n)=\begin{pmatrix}
    0 & -1 \\
    p^n & 0
\end{pmatrix}, z(p^n)=\begin{pmatrix}
    p^n & 0 \\
    0 & p^n
\end{pmatrix},  x(t)=\begin{pmatrix}
    1 & t \\
    0 & 1
\end{pmatrix},  y(s)=\begin{pmatrix}
    1 & 0 \\
    s & 1
\end{pmatrix}$. For any $g\in G$ we denote by $X_g$ the characteristic function of the double coset $KgK$ and define the Hecke Algebra operators: $$\mathcal{T}_n=X_{d(p^n)}, \mathcal{U}_n=X_{w(p^n)},\mathcal{Z}=X_{z(p)}.$$ The following description of the Iwahori Hecke algebra is then well known.
\begin{theorem} \label{struct} (see \cite[Theorem 5]{BarSom})
The Hecke algebra $\mathcal{H}(G/\!\!/K_0(p),1)$ is generated by $\mathcal{U}_0, \mathcal{U}_1$ and $\mathcal{Z}$ with relations:
	\begin{enumerate}
		\item $\mathcal{U}_1^2=\mathcal{Z}$
		\item $(\mathcal{U}_0+1)(\mathcal{U}_0-p)=0$
		\item $\mathcal{Z}$ commutes with $\mathcal{U}_0,\mathcal{U}_1$
	\end{enumerate}
\end{theorem}
We are primarily interested in $\mathcal{U}_0, \mathcal{U}_1, \mathcal{T}_1$ and specifically their relation:
$$\mathcal{U}_0=\mathcal{Z}^{-1}*\mathcal{T}_1*\mathcal{U}_1.$$

\subsection{The Subalgebra $\mathcal{H}(K/\!\!/K_0(p^n),\chi_p)$} \label{TwistHecke}
It is difficult and maybe impossible to obtain a full description of the Hecke algebra $\mathcal{H}(G/\!\!/K_0(p^n),\chi_p)$ for general $n$. 
Following Baruch and Purkait we consider the subalgebra:
$$\mathcal{H}(K/\!\!/K_0(p^n),\chi_p)=\{f:K\to\mathbb{C}|f(kxk')=\chi_p(kk')f(x),\forall x\in K, k,k'\in K_0(p^n)\}.$$
Baruch and Purkait obtained generators and relations for the subalgbera \\ $\mathcal{H}(K/\!\!/K_0(p^n),1)$ and used it to describe representations of $K$ having $K_0(p^n)$ fixed vectors which were first described by Casselman using a different method. We now proceed to the case where $\chi_{p} \not \equiv 1$ .
Let $\chi_p$ be a $p-$adic character of conductor $1+p^r\mathbb{Z}_p$ with $r\leq n$. Consider the elements $y(p^j)=\begin{pmatrix}
    1 & 0 \\
    p^j & 1
\end{pmatrix}$ and the corresponding functions:\\ $\mathcal{V}_j^{\chi_p}(x)= \begin{cases}
        \chi_p(k_1k_2) & \text{, if } x=k_1y(p^j)k_2\\
        0 & \text{, otherwise}
    \end{cases}$, where $k_1,k_2\in K_0(p^n)$ and \\ $\mathcal{U}_0^{\chi_p}(x)= \begin{cases}
        \chi_p(k_1k_2) & \text{, if } x=k_1w(1)k_2\\
        0 & \text{, otherwise}
    \end{cases}$ \\ \\
\par
By \cite[Lemma 1]{Cass} we have the following lemma:
 \begin{lemma}\label{doubcos}
 	The elements $w(1),y(p),...,y(p^{n-1}),y(p^n)$ form a complete set of double coset representatives for $K/\!\!/K_0(p^n)$. 
 \end{lemma}
 We say an element $g\in K$ is supported on $\mathcal{H}(K/\!\!/K_0(p^n),\chi_p)$ if the characteristic function $X_g^{\chi_p}(x)= \begin{cases}
        \chi_p(k_1k_2) & \text{, if } x=k_1gk_2\\
        0 & \text{, otherwise}
    \end{cases}$ where $k_1,k_2\in K_0(p^n)$, is well defined. We now examine which of the elements in Lemma \ref{doubcos} are supported on the algebra $\mathcal{H}(K/\!\!/K_0(p^n),\chi_p)$. The following lemma is well known:
 \begin{lemma} \label{supchar}
An element $g\in K$ is supported on $\mathcal{H}(K/\!\!/K_0(p^n),\chi_p)$ if and only if $\chi_p(gkg^{-1}k^{-1})=1, \forall k\in g^{-1}K_0(p^n)g\cap K_0(p^n)=K_g$.
 \end{lemma}
 \begin{lemma} \label{nilentr}
     For $g=y(p^m)$ we have that $K_g=\{\begin{pmatrix}
    a & b \\
    c & d
\end{pmatrix}\in K_0(p^n):a-d+p^mb\in p^{n-m}\mathbb{Z}_p\}$ and $K_{w(1)}=\{\begin{pmatrix}
    a & b \\
    c & d
\end{pmatrix}\in K_0(p^n):b\in p^n\mathbb{Z}_p\}$. Notice that $K_{w(1)}$ contains the diagonal subgroup of $K$.
 \end{lemma}
 \begin{proof}
     Writting $y(-p^m)\begin{pmatrix}
    a & b \\
    c & d
\end{pmatrix}y(p^m)=\begin{pmatrix}
    a+p^mb & b \\
    -p^ma+c-p^{2m}b+p^md & d-p^mtb
\end{pmatrix}$ and requiring the lower left entry to be in $p^n\mathbb{Z}_p$ implies the result. The proof for $w(1)$ is similar.
 \end{proof}
 \begin{lemma}
 	Assume $n\geq r>0$. The Hecke algebra $\mathcal{H}(K/\!\!/K_0(p^n),\chi_p)$ is supported exactly on $y(p^r),y(p^{r+1}),...,y(p^n)$.  
 \end{lemma}
 \begin{proof}
We check each element individually:
\begin{itemize}
    \item For $g=w(1)$ we have $\chi_p(w(1)\begin{pmatrix}
    a & 0 \\
    0 & b
\end{pmatrix}w(1)^{-1}\begin{pmatrix}
    a^{-1} & 0 \\
    0 & b^{-1}
\end{pmatrix})=\chi_p(ab^{-1})$ and thus, since $a,b$ can be any units in $\mathbb{Z}_p^{\times}$, there is no support on $w(1)$.
    \item For $g=y(p^m)$ and $k=\begin{pmatrix}
    a & b \\
    c & d
\end{pmatrix}\in K_{y(p^m)}$, we have that $gkg^{-1}k^{-1}=\begin{pmatrix}
    1 & 0 \\
    p^m & 1
\end{pmatrix}\begin{pmatrix}
    a & b \\
    c & d
\end{pmatrix}\begin{pmatrix}
    1 & 0 \\
    -p^m & 1
\end{pmatrix}\begin{pmatrix}
    a & b \\
    c & d
\end{pmatrix}^{-1}=\begin{pmatrix}
    * & * \\
    * & \frac{bp^m(p^mb+d)}{det(k)}+1
\end{pmatrix}$
Picking $\begin{pmatrix}
    1-p^mb & b \\
    0 & 1
\end{pmatrix}\in K_0(p^n)$ we have that the condition of Lemma \ref{nilentr} is satisfied and thus $\begin{pmatrix}
    1-p^mb & b \\
    0 & 1
\end{pmatrix}\in K_{y(p^m)}$. But then we get that $\chi(gkg^{-1}k^{-1})=\chi(\begin{pmatrix}
    * & * \\
    * & s'(b)p^m+1
\end{pmatrix})=\chi(s'(b)p^m+1)$ with $s'(b)=\frac{b(p^mb+1)}{(1-p^mb)}$. For any $\xi\in\mathbb{Z}_p$ we want to show that $s'(b)=\xi$ has a solution in $\mathbb{Z}_p$ or equivalently that $f(b)=p^mb^2+(1+p^m\xi)b-\xi$ has a root. But $f(\xi)\in p\mathbb{Z}_p$ and $f'(\xi)\not\in p\mathbb{Z}_p$ which implies by Hensel's Lemma the existence of a $b\in \mathbb{Z}_p$ such that $f(b)=0$. We may thus replace $s'(b)$ by any $\xi\in\mathbb{Z}_p$ and obtain that $\chi(gkg^{-1}k^{-1})=\chi(\xi p^m+1)$ with $\xi$ free in $\mathbb{Z}_p$. \\
If $p^r|p^m$ or equivalently $r\leq m$, then this implies support at $y(p^m)$ since $\chi(p^r\mathbb{Z}_p+1)=\chi(1)=1$. For the other direction: if $\chi_p(p^m\xi+1)=1,\forall \xi\in\mathbb{Z}_p$ and $p^r\nmid p^m$ then $\chi_p$ of conductor $1+p^r\mathbb{Z}_p$ is induced by a character of $\mathbb{Z}_p^{\times}/(1+p^m\mathbb{Z}_p)$ with $m<r$, a contradiction. Thus we have support at $y(p^m) \iff r\leq m$.
\end{itemize}

 \end{proof}
It is thus evident that $\mathcal{H}(K/\!\!/K_0(p^n),\chi_p)$ is spanned by $\mathcal{V}_r^{\chi_p},..., \mathcal{V}_{n}^{\chi_p}$ if $\chi_p$ is non trivial and $\mathcal{V}_r^{1},..., \mathcal{V}_{n}^{1}$ and $\mathcal{U}_{0}^{1}$ if $\chi_p\equiv 1$ (see \cite[Theorem 6]{BarSom}). Notice that $\mathcal{V}_n^{\chi_p}$ is supported on $K_0(p^n)$ and is simply the identity of the algebra. We also use the following easy to prove lemmas:
\begin{lemma}  \label{conv}
Let $f_1,f_2\in\mathcal{H}(K/\!\!/K_0(p^n),\chi_p)$ such that $f_1$ is supported on the double coset $K_0(p^n)xK_0(p^n)= \bigcup\limits_{i=1}^ma_iK_0(p^n)$ and similarly $f_2$ is supported on $K_0(p^n)yK_0(p^n)= \bigcup\limits_{i=1}^kb_iK_0(p^n)$. Then the convolution product is $$f_1*f_2(h)=\sum\limits_{i=1}^mf_1(a_i)f_2(a_i^{-1}h)=\sum\limits_{i=1}^kf_1(hb_i)f_2(b_i^{-1}).$$ In order to have support at a value $h$ it is necessary to have $h\in a_ib_jK_0(p^n)$ for some $b_j$.
\end{lemma}
\begin{lemma} \label{sup}
	If $r\leq i,j$ then $supp(\mathcal{V}_i^{\chi_p}*\mathcal{V}_j^{\chi_p}) \subseteq supp(\mathcal{V}_i^{1}*\mathcal{V}_j^{1})$.
\end{lemma}
\begin{proof}
	Immediate by noting that $|\mathcal{V}_j^{\chi_p}(x)|=\mathcal{V}_j^{1}(x), \forall j\geq r$.
\end{proof}
Also from Lemma 3.7 and Lemma 3.9 of \cite{BarSom}:
\begin{lemma} \label{decomp}
 For every $1\leq j\leq n-1$ it holds that
	$K_0(p^n)y(p^j)K_0(p^n)= \bigcup\limits_{s\in\mathbb{Z}_p^{\times}/1+p^{n-j}\mathbb{Z}_p}d(s)y(p^j)K_0(p^n)= \bigcup\limits_{s\in\mathbb{Z}_p^{\times}/1+p^{n-j}\mathbb{Z}_p}K_0(p^n)y(p^j)d(s)$
\end{lemma}
We can now describe the relations between the elements we defined above. 
\begin{prop}
	Inside $\mathcal{H}(K/\!\!/K_0(p^n),\chi_p)$ we have the following relations:
	\begin{enumerate}
		\item \[\mathcal{V}_{\ell}^{\chi_p}*\mathcal{V}_{\ell}^{\chi_p}=p^{n-\ell-1}(p-1)\sum_{i=\ell+1}^n\mathcal{V}_i^{\chi_p}+p^{n-\ell-1}(p-2)\mathcal{V}_{\ell}^{\chi_p}\]
		\item \[\mathcal{V}_{\ell}^{\chi_p}*\mathcal{V}_j^{\chi_p}=p^{n-j-1}(p-1)\mathcal{V}_{\ell}^{\chi_p}=\mathcal{V}_j^{\chi_p}*\mathcal{V}_{\ell}^{\chi_p}\] for $\ell<j<n$
		\item If $\mathcal{Y}_{\ell}^{\chi_p}=\sum\limits_{i=\ell}^n\mathcal{V}_i^{\chi_p}$ then \[\mathcal{V}_{\ell}^{\chi_p}*\mathcal{Y}_{\ell+1}^{\chi_p}=p^{n-\ell-1}\mathcal{V}_{\ell}^{\chi_p}=\mathcal{Y}_{\ell+1}^{\chi_p}*\mathcal{V}_{\ell}^{\chi_p}\] and thus from $1.$ we have: \[(\mathcal{V}_{\ell}^{\chi_p}-p^{n-\ell-1}(p-1))(\mathcal{V}_{\ell}^{\chi_p}+\mathcal{Y}_{\ell+1}^{\chi_p})=0\]
	\end{enumerate}
\end{prop}
\begin{proof}
\begin{enumerate}
	\item From \cite[Proposition 3.10]{BarSom} we know that $\mathcal{V}_{\ell}^{1}*\mathcal{V}_{\ell}^{1}$ is supported exactly on $K_0(p^n)y(p^i)K_0(p^n)$ with $i\geq\ell$ and thus $\mathcal{V}_{\ell}^{\chi_p}*\mathcal{V}_{\ell}^{\chi_p}=\sum\limits_{i=\ell}^n\lambda_i\mathcal{V}_{i}^{\chi_p}$ with every $ \lambda_i\in\mathbb{C}$. It remains then to compute $\lambda_i=\mathcal{V}_{\ell}^{\chi_p}*\mathcal{V}_{\ell}^{\chi_p}(y(p^i))$. From Lemmas \ref{conv}, \ref{decomp} and noting that $\chi_p(d(s))=1$ we get: \[\mathcal{V}_{\ell}^{\chi_p}*\mathcal{V}_{\ell}^{\chi_p}(1)=\sum_{s\in\mathbb{Z}_p^{\times}/1+p^{n-\ell}\mathbb{Z}_p}\mathcal{V}_{\ell}^{\chi_p}(y(-p^{\ell}))\] And since $(\alpha):\begin{pmatrix}
    -1 & 0 \\
    0 & 1
\end{pmatrix}\begin{pmatrix}
    1 & 0 \\
    p^{\ell} & 1
\end{pmatrix}\begin{pmatrix}
    -1 & 0 \\
    0 & 1
\end{pmatrix}=\begin{pmatrix}
    1 & 0 \\
    -p^{\ell} & 1
\end{pmatrix},$ \[\mathcal{V}_{\ell}^{\chi_p}*\mathcal{V}_{\ell}^{\chi_p}(1)=\#|\mathbb{Z}_p^{\times}/1+p^{n-\ell}\mathbb{Z}_p|=p^{n-\ell-1}(p-1).\] Similarly for $j>\ell$: \[\mathcal{V}_{\ell}^{\chi_p}*\mathcal{V}_{\ell}^{\chi_p}(y(p^j))=\sum_{s\in\mathbb{Z}_p^{\times}/1+p^{n-j}\mathbb{Z}_p}\mathcal{V}_{\ell}^{\chi_p}(y(-p^{\ell})d(s)y(p^{j}))\] The only non-zero terms in the sum are those for which $\exists k_1, k_2\in K_0(p^n):y(-p^{\ell})d(s)y(p^{j})=k_1y(p^{\ell})k_2$. Taking $k_2=\begin{pmatrix}
    \frac{p^{j-\ell}-s^{-1}}{p^{n-\ell}-1} & 0 \\
    0 & -1
\end{pmatrix}$ it is easy to see that $k_1=y(-p^{\ell})d(s)y(p^{j})k_2^{-1}y(-p^{\ell})=\begin{pmatrix}
    * & 0 \\
    * & -1 \end{pmatrix}$ and thus $\chi_p(k_1k_2)=\chi_p(-1)^2=1$ giving us again: \[\mathcal{V}_{\ell}^{\chi_p}*\mathcal{V}_{\ell}^{\chi_p}(y(p^j))=\#|\mathbb{Z}_p^{\times}/1+p^{n-\ell}\mathbb{Z}_p|=p^{n-\ell-1}(p-1).\] 
    For $j=\ell$ we need $k_1, k_2\in K_0(p^n):k_1=y(-p^{\ell})d(s)y(p^{j})k_2^{-1}y(-p^{\ell})$. Setting $k_2^{-1}=\begin{pmatrix}
    a & b \\
    c & d
\end{pmatrix}$ we thus require $(1-s^{-1})(a-bp^{\ell})-d\in p^{n-\ell}\mathbb{Z}_p$ and since $d\in\mathbb{Z}_p^{\times}$ this holds exactly when $s\not\in 1+p\mathbb{Z}_p$ in which case we can choose $k_2=\begin{pmatrix}
      \frac{1-s^{-1}}{p^{n-\ell}-1} & 0 \\
     0 & -1
\end{pmatrix}$ and we get as before $\chi_p(k_1k_2)=1$. This leads to: \[\mathcal{V}_{\ell}^{\chi_p}*\mathcal{V}_{\ell}^{\chi_p}(y(p^{\ell}))=\#\{s\in\mathbb{Z}_p^{\times}/1+p^{n-\ell}\mathbb{Z}_p:s\not\in 1+p\mathbb{Z}_p\}=p^{n-\ell-1}(p-2).\]

	\item From Lemma \ref{sup} and the relation $\mathcal{V}_{\ell}^{1}*\mathcal{V}_j^{1}=p^{n-j-1}(p-1)\mathcal{V}_{\ell}^{1}$ as shown in \cite{BarSom} it follows that we only need to compute \[\mathcal{V}_{\ell}^{\chi_p}*\mathcal{V}_{j}^{\chi_p}(y(p^{\ell}))=\sum_{s\in\mathbb{Z}_p^{\times}/1+p^{n-\ell}\mathbb{Z}_p}\mathcal{V}_{j}^{\chi_p}(y(-p^{\ell})d(s)y(p^{\ell}))\] So similar to 1. we need to count the $s$ for which $\exists k_1, k_2\in K_0(p^n):k_1=y(-p^{\ell})d(s)y(p^{\ell})k_2^{-1}y(-p^j)$ which is equivalent to requiring $(1-s^{-1})(a-bp^j)-p^{j-\ell}d\in p^{n-\ell}\mathbb{Z}_p$ implying $s\in 1+p^{j-\ell}\mathbb{Z}_p$. In that case setting $s=1+p^{j-\ell}u, u\in\mathbb{Z}_p^{\times}$ we take $k_2=\begin{pmatrix}
    s^{-1} & 0 \\
    0 & u^{-1}
\end{pmatrix}$ and observing that $k_1=\begin{pmatrix}
    * & 0 \\
    * & u
\end{pmatrix}$ implies as before that $\chi_p(k_1k_2)=\chi_p(u)\chi_p(u^{-1})=1$ and so: \[\mathcal{V}_{\ell}^{\chi_p}*\mathcal{V}_{j}^{\chi_p}(y(p^{\ell}))=\#|s\in\mathbb{Z}_p^{\times}/1+p^{n-\ell}\mathbb{Z}_p:s\in 1+p^{j-\ell}\mathbb{Z}_p|=p^{n-j-1}(p-1).\] The case $\mathcal{V}_{j}^{\chi_p}*\mathcal{V}_{\ell}^{\chi_p}$ follows similarly using the second equality of Lemma \ref{conv} and $(\alpha)$ to get $\mathcal{V}_{\ell}^{\chi_p}(y(-p^{\ell}))=\mathcal{V}_{\ell}^{\chi_p}(k_0y(p^{\ell})k_0)=1$ for $k_0=\begin{pmatrix}
    -1 & 0 \\
    -p^n & 1
\end{pmatrix}$, which then imply: \[\mathcal{V}_{j}^{\chi_p}*\mathcal{V}_{\ell}^{\chi_p}(y(p^{\ell}))=\sum_{s\in\mathbb{Z}_p^{\times}/1+p^{n-\ell}\mathbb{Z}_p}\mathcal{V}_{j}^{\chi_p}(y(p^{\ell})d(s)y(-p^{\ell}))\]
	\item This follows immediately from 1. and 2.
	\end{enumerate}
	\end{proof}
The same calculations from \cite{BarSom} now give us the full description for the relations of the algebra. 
\begin{prop}
It holds that: $\mathcal{Y}_{j}^{\chi_p}*\mathcal{Y}_{\ell}^{\chi_p}=\mathcal{Y}_{\ell}^{\chi_p}*\mathcal{Y}_{j}^{\chi_p}=p^{n-j}\mathcal{Y}_{\ell}^{\chi_p}, \forall j: j\geq\ell\geq r$. Furthermore if $\chi_p\equiv 1$ then we also have the following relations:
\begin{enumerate}
	\item $\mathcal{U}_0^1*\mathcal{U}_0^1=p^{n-1}(p-1)\mathcal{U}_0^1+p^n\mathcal{Y}^1_1$
	\item $\mathcal{U}_0^1*\mathcal{Y}_{\ell}^{1}=\mathcal{Y}_{\ell}^{1}*\mathcal{U}_0^1=p^{n-\ell}\mathcal{U}_0^1$
	\item $\mathcal{U}_0^1*(\mathcal{U}_0^1-p^n)*(\mathcal{U}_0^1+p^{n-1})=0$
\end{enumerate}
\end{prop}
Combining this with \cite[Proposition 3.12]{BarSom} for the case $\chi_p\equiv 1$ gives us the following result about the structure of $\mathcal{H}(K/\!\!/K_0(p^n),\chi_p)$:
\begin{theorem}
	Let $\chi_p$ be a character of conductor $1+p^r\mathbb{Z}_p$. Then the Hecke algebra $\mathcal{H}(K/\!\!/K_0(p^n),\chi_p)$ is an $n-r+1$ dimensional commutative algebra generated by $\mathcal{Y}_{r}^{\chi_p}, \mathcal{Y}_{r+1}^{\chi_p},...,\mathcal{Y}_{n}^{\chi_p}$ for $r> 0$ and $\mathcal{Y}_{1}^{1}, \mathcal{Y}_{2}^{1},...,\mathcal{Y}_{n}^{1},\mathcal{U}_0^1$ for $r=0$.
\end{theorem}

\subsection{Representations of $K$ with $(\chi_p,K_0(p^n))$-fixed vectors}
Throughout this subsection let $\chi_p$ be a $p-$adic character  of conductor $1+p^r\mathbb{Z}_p$. We extend $\chi_p$ to a character of $K_0(p^n)$ by setting (abusing notation) $\chi_p(A)=\chi_p(\begin{pmatrix}
    a & b \\
    c & d
\end{pmatrix})=\chi_p(d)$ for $A\in K_0(p^n), n\geq r$. We only examine the case where $\chi_p\not\equiv 1$, that is $r>0$.For the trivial character case $\chi_p\equiv 1$ we refer the reader to \cite{BarSom}. Consider now the space \[I(n)=Ind_{K_0(p^n)}^K\chi_p=\{\phi:K\to\mathbb{C}:\phi(k_0k)=\chi_p(k_0)\phi(k),\forall k\in K, k_0\in K_0(p^n)\}.\]
It is known that $I(n)$ is a left $K$-module with action given by right translation $\pi_R(k)\phi(x)=\phi(xk)$ and of dimension $|K:K_0(p^n)|dim(\chi_p)=p^{n-1}(p+1)$. 
\begin{defi}
	Let $V$ be a complex vector space and $\pi:K\to V$ a representation. We call a vector $v\in V, (\chi_p,K_0(p^n))$-fixed if $\pi(k)v=\chi_p(k)v,\forall k\in K_0(p^n)$.
\end{defi}
\begin{prop}
	There exists a correspondence between irreducible subrepresentations of $I(n)$ and smooth irreducible representations of $K$ containing a $(\chi_p,K_0(p^n))$-fixed vector. 
\end{prop} 
\begin{proof}
	By Frobenius Reciprocity it follows that if $\rho:K\to V$ is a smooth irreducible representation \[Hom_K(I(n),V)= Hom_{K_0(p^n)}(\mathbb{C},V)= Hom_{K_0(p^n)}(\chi_p,\rho|_{K_0(p^n)}).\]
	If $\phi\in Hom_{K_0(p^n)}(\mathbb{C}_{\chi_p},V)$ with $\phi(1)=v$ then $\chi_p(k_0)v=\phi(\chi_p(k_0)\cdot 1)=\rho(k_0)v$ completing the proof.  
	\end{proof}

We thus consider the space of $(\chi_p,K_0(p^n))$-fixed vectors \[I(n)^{(\chi_p,K_0(p^n))}=\{f\in I(n):\rho(k_0)f=\chi_p(k_0)f,\forall k_0\in K_0(p^n)\}\] and immediately observe that \[I(n)^{(\chi_p,K_0(p^n))}=\{f:K\to\mathbb{C}:f(kxk')=\chi_p(kk')f(x), \forall k,k'\in K_0(p^n),x\in K\}.\] From this we conclude that $I(n)^{(\chi_p,K_0(p^n))}=\mathcal{H}(K/\!\!/K_0(p^n),\chi_p)$ which means that $dim(I(n)^{(\chi_p,K_0(p^n))})=n-r+1$. \\
Observing that $End_K(I(n))\cong\mathcal{H}(K/\!\!/K_0(p^n),\chi_p)$ it follows that $I(n)$ is multiplicity free if and only if $\mathcal{H}(K/\!\!/K_0(p^n),\chi_p)$ is commutative. This means that $I(n)$ is multiplicity free which implies that the correspondence between smooth irreducible representations of $K$ and $(\chi_p,K_0(p^n))$-fixed vectors is bijective. Combining these facts then we obtain that:
\begin{prop}
	The representation $I(n)$ is a sum of $n+1-r$ distinct irreducible representations. 
\end{prop}
\begin{proof}
	It is immediate from the above argument that $I(n)$ has exactly \\$dim(I(n)^{(\chi_p,K_0(p^n))})=n+1-r$ irreducible component and is multiplicity free. 
\end{proof}
Let $\sigma(k)$ be an irreducible representation of $K$ such that $\sigma(k)$ contains an $(\chi_p,K_0(p^k))$-fixed vector but not a $(\chi_p,K_0(p^{k-1}))$-fixed vector.
Observe that since $I(r)$ has exactly $1$ irreducible representation we obtain the equality $I(r)=\sigma(r)$ and $dim(\sigma(r))=dim(I(r))=p^{r-1}(p+1)$. Now $I(r+1)$ contains exactly a $(\chi_p,K_0(p^{r+1}))$-fixed vector that is not $(\chi_p,K_0(p^{r}))$-fixed which corresponds to $\sigma(r+1)$ and $dim(\sigma(r+1))=dim(I(r+1))-dim(I(r))=p^{r-1}(p^2-1)$. Continuing in this manner we obtain for each $k\in\{r,...,n\}$ a unique $(\chi_p,K_0(p^{k}))$-fixed vector that is not $(\chi_p,K_0(p^{k-1}))$-fixed. 
\begin{cor} 
	Let $k\geq r$, then there exists a unique irreducible representation $\sigma(k)$ containing a unique up to scalar $(\chi_p,K_0(p^{k}))$-fixed vector that is not $(\chi_p,K_0(p^{k-1}))$-fixed and $dim(\sigma(r))=p^{r-1}(p+1), dim(\sigma(k))=p^{k-2}(p^2-1)$ for $k\in\{r+1,...,n\}$.
\end{cor}
The result is partially known due to Casselman in a more general setting:
\begin{theorem}[\textbf{Casselman} \cite{CassLehn} Theorem 1]
    Let $\rho$ be an irreducible admissible infinite dimensional representation of $G$ where $\rho=\chi_p$ on the scalar matrices. If $1+p^{\mathfrak{c}(\rho)}\mathbb{Z}_p$ is the largest ideal of $\mathbb{Z}_p$ such that there exists a $(\chi_p,K_0(p^{\mathfrak{c}(\rho)}))$-fixed 
 vector $v$, then $v$ is unique up to scalar multiples. 
\end{theorem}

We can now explicitly describe these irreducible representations of $K$. Consider the left action of $\mathcal{H}(K/\!\!/K_0(p^n),\chi_p)$ on $I(n)$ given by $\pi_L(\phi)(f)(x)=\phi*f(x)$ with $f\in I(n), \phi \in \mathcal{H}(K/\!\!/K_0(p^n),\chi_p)$. Since this action commutes with $\pi_R$ it follows from Schur's Lemma that $\mathcal{H}(K/\!\!/K_0(p^n),\chi_p)$ acts via scalar multiplication on the irreducible components of $I(n)$. Let $\sigma$ be such a component which contains the fixed vector $v_{\sigma}$, then $v_{\sigma}$ is a common eigenvector for every basis element of $\mathcal{H}(K/\!\!/K_0(p^n),\chi_p)$. We compute these eigenvectors for each $\sigma(k), k\in\{r,...,n\}$. 
\begin{prop}
	A basis of eigenvectors for $\mathcal{H}(K/\!\!/K_0(p^n),\chi_p)$ under the action $\pi_L$ when $r>0$ is:\\$v_r=\mathcal{Y}_{r}^{\chi_p},$\\ $v_{r+1}=\mathcal{Y}_r^{\chi_p}-p\mathcal{Y}_{r+1}^{\chi_p},$ \\ \vdots \\ $v_{n}=\mathcal{Y}_{n-1}^{\chi_p}-p\mathcal{Y}_{n}^{\chi_p}$.
	\end{prop}
 The eigenvalues of the above action are given in the table bellow where the entry at the place $v_i,\mathcal{Y}_j^{\chi_p}$ is the eigenvalue $\lambda$ of the action of $\mathcal{Y}_j^{\chi_p}$ on $v_i$ i.e. $\mathcal{Y}_j^{\chi_p}(v_i)=\lambda v_i$. For example: $\mathcal{Y}_{r+k}^{\chi_p}(v_k)=p^{n-r-k} v_k$.

\begin{table}[ht]
\begin{tabular}{c|ccccccccc}
    & $\mathcal{Y}^{\chi_p}_r$ & $\mathcal{Y}^{\chi_p}_{r+1}$ & $\mathcal{Y}^{\chi_p}_{r+2}$ & $\mathcal{Y}^{\chi_p}_{r+3}$ & $\hdots$ & $\mathcal{Y}^{\chi_p}_{r+k}$ & $\hdots$ & $\mathcal{Y}^{\chi_p}_{n-1}$ & $\mathcal{Y}^{\chi_p}_{n}$  \\
[0.5ex]
\hline
$v_r$ & $p^{n-r}$ & $p^{n-r-1}$ & $p^{n-r-2}$ & $p^{n-r-3}$ & $\hdots$ & $p^{n-r-k}$ & $\hdots$ & $p$ & $1$ \\
$v_{r+1}$ & $0$ & $p^{n-r-1}$ & $p^{n-r-2}$ & $p^{n-r-3}$ & $\hdots$ & $p^{n-r-k}$ & $\hdots$ & $p$ & $1$ \\
$\vdots$ & $\vdots$ & $\vdots$ & $\vdots$ & $\vdots$ & $\vdots$ & $\vdots$ & $\vdots$ & $\vdots$ & $\vdots$ \\
$v_{r+k}$ & $0$ & $0$ & $0$ & $0$ & $\hdots$ & $p^{n-r-k}$ & $\hdots$ & $p$ & $1$ \\
$\vdots$ & $\vdots$ & $\vdots$ & $\vdots$ & $\vdots$ & $\vdots$ & $\vdots$ & $\vdots$ & $\vdots$ & $\vdots$ \\
$v_{n-1}$ & $0$ & $0$ & $0$ & $0$ & $\hdots$ & $0$  & $\hdots$ & $p$ & $1$ \\
$v_n$ & $0$ & $0$ & $0$ & $0$ & $\hdots$ & $0$ & $\hdots$ & $0$ & $1$
\end{tabular}
\end{table}

If we instead consider the eigenvalues of the action on the basis consisting of the elements $\mathcal{V}^{\chi_p}_j$, then with similar notation as above we get the table:

\begin{table}[ht]
\begin{tabular}{c|ccccccccc}
    & $\mathcal{V}^{\chi_p}_r$ & $\mathcal{V}^{\chi_p}_{r+1}$ & $\mathcal{V}^{\chi_p}_{r+2}$ & $\hdots$ & $\mathcal{V}^{\chi_p}_{n-1}$ & $\mathcal{V}^{\chi_p}_{n}$  \\
[0.5ex]
\hline
$v_r$ & $p^{n-r-1}(p-1)$ & $p^{n-r-2}(p-1)$ & $p^{n-r-3}(p-1)$  & $\hdots$  & $p-1$ & $1$ \\
$v_{r+1}$ & $-p^{n-r-1}$ & $p^{n-r-2}(p-1)$ & $p^{n-r-3}(p-1)$ & $\hdots$  & $p-1$ & $1$ \\
$v_{r+2}$ & $0$ & $-p^{n-r-2}$ & $p^{n-r-3}(p-1)$  & $\hdots$   & $p-1$ & $1$ \\
$\vdots$ & $\vdots$  & $\vdots$  & $\vdots$ & $\hdots$ & $\vdots$ & $\vdots$ \\
$v_{n-1}$ & $0$  & $0$ & $\vdots$ & $\hdots$ & $p-1$ & $1$ \\
$v_n$ & $0$  & $0$ & $\vdots$ & $\hdots$ & $-1$ & $1$
\end{tabular}
\end{table}
\begin{lemma}
    The operators $\pi_L(\mathcal{V}^{\chi_p}_j)$ with $r \leq j\leq n-1$ have trace $0$. 
\end{lemma}
\begin{proof}
 Let $\phi_g(x)=\begin{cases}
        \chi_p(k_0) & \text{, if } x=k_0g, k_0\in K_0(p^n)\\
        0 & \text{, otherwise}
    \end{cases}$, then these functions as $g$ runs through the coset representatives $y(p^i)$ form a basis of $I(n)$. It is sufficient then to prove that $\pi_L(\mathcal{V}^{\chi_p}_j)(\phi_g)(g)=0$. But $\pi_L(\mathcal{V}^{\chi_p}_j)(\phi_g)$ can only have support on $K_0(p^n)y(p^j)K_0(p^n)g$ and it is obvious that $g\not\in K_0(p^n)y(p^j)K_0(p^n)g$ which means $\pi_L(\mathcal{V}^{\chi_p}_j)(\phi_g)=0$. 
\end{proof}  
The above lemma and an observation of the table lead to the following conclusion for $r>1$:
\begin{cor}
The representation $I(n)$ is a sum of $n+1-r$ irreducible subspaces given by $T_k=Span(\pi_R(K)v_k), k\in\{r,...,n\}$. It also follows that $dim(T_k)=dim(\sigma(k)), k\in\{r,...,n\}$. It follows that $T_k=\sigma(k)$ since it is the unique irreducible subspace containing a unique $(\chi_p,K_0(p^k))-$fixed vector $u_k$ that is not $(\chi_p,K_0(p^{k-1}))-$fixed.
\end{cor}
\begin{proof}
    From the table above, since the operators in question have trace $0$ we get the system of equations:
    \[(p-1)d_r-d_{r+1}=0\]
    \[(p-1)d_r+(p-1)d_{r+1}-d_{r+2}=0\] 
    \[\vdots\]
    \[(p-1)(d_r+\hdots+d_{n-1})-d_n=0\]
    \[d_r+\hdots+d_n=p^{n-1}(p+1)\]
    Where $d_k=dim(T_k)$. Solving this system we get the corresponding dimensions which match those of $\sigma(k)$. Noting that each vector $u_k$ is in a different irreducible component and $u_k\in T_k$ proves the claim.
\end{proof}

\section{From Adelic to Classical} \label{AdeltoCl}
Let $\mathbb{A}$ be the adele ring of $\mathbb{Q}$ and consider $G_{\mathbb{A}}=GL_2(\mathbb{A})$ and its center $Z_{\mathbb{A}}$.
We begin this section by stating the Strong Approximation Theorem from which we have: $G_{\mathbb{A}}=GL_2(\mathbb{Q})GL_2(\mathbb{R})^+\times K_0(N)$, where $K_0(N)=\prod\limits_{p<\infty}K_p$ with $K_p=GL_2(\mathbb{Z}_p)$ when $p\nmid N$ and $K_p=K_0(p^{n_p})$ at each place satisfying $p^{n_p}\mid\!\mid N$. Thus any $g\in G_{\mathbb{A}}$ may be written as $g=qg_{\infty}k_0$ where $q\in GL_2(\mathbb{Q}), g_{\infty}\in GL_2(\mathbb{R})^{+}$ and $k_0\in K_0(N)$ (meaning the corresponding projections inside the adeles). Moreover let $\chi_{id}$ be the idelic lift to $K_0(N)$ of the $\mod N$ Dirichlet character $\chi$.
\begin{defi}
    We denote by $A_k(N,\chi_{id})$ the space of functions \\ $\phi:GL_2(\mathbb{Q})\backslash G_{\mathbb{A}}\to\mathbb{C}$ satisfying:
    \begin{enumerate}
        \item $\phi(gk_0)=\phi(g)\chi_{id}(k_0), \forall g\in G_{\mathbb{A}}, k_0\in K_0(N)$
        \item $\phi(gr(\theta))=e^{-ik\theta}\phi(g), r(\theta)=\begin{bmatrix}
cos(\theta) & sin(\theta) \\
-sin(\theta) & cos(\theta) 
\end{bmatrix}$
\item $\Delta\phi=\frac{k}{2}(1-\frac{k}{2})\phi$
\item $\phi(zg)=\phi(g)\chi(z)$ for all $z\in Z_{\mathbb{A}}$
\item $\phi$ is slowly increasing
\item $\phi$ is cuspidal i.e. $\int_{\mathbb{Q}\backslash A}\phi(\begin{bmatrix}
1 & x \\
0 & 1 
\end{bmatrix}g)dx=0$.
    \end{enumerate}
\end{defi}
From \cite[(3.4)]{Gel} we obtain the following isomorphism of this space into the space of classical cuspforms:
\begin{prop}
        There exists an isomorphism $q:A_k(N,\chi_{id})\to S_{k}(\Gamma_0(N),\chi)$ given by sending $\phi \in A_k(N,\chi_{id})$ to $f\in S_k(\Gamma_0(N),\chi)$ where 
     \[\phi(g)=f(g_{\infty}i)j(g_{\infty},i)^{-k}\chi_{id}(k_0)\] and 
     \[f(z)=\phi(g_{\infty})j(g_{\infty},i)^k.\]
    \end{prop}
Let $\phi\in A_k(N,\chi_{id})$ then $\Phi\in \mathcal{H}(G/\!\!/K_0(p^n),\bar{\chi}_{id})$ acts on $\phi$ via: \[\Phi(\phi)(x)=\int_{G}\Phi(g)\phi(xg)dg.\] This implies that $\mathcal{H}(G/\!\!/K_0(p^n),\bar{\chi}_{id})\subseteq End_{\mathbb{C}}(A_k(N,\chi))$. In fact $q$ induces a map $q:End_{\mathbb{C}}(A_k(N,\chi_{id}))\to End_{\mathbb{C}}(S_k(\Gamma_0(N),\chi))$. The map explicitly is \[q(\Phi)(f)=q(\Phi(\phi)).\] We call $q$ the de-adelization of a $p-$adic Hecke operator $\Phi$ to a classical Hecke operator $q(\Phi)$.
\begin{lemma}
Let $\chi$ be a Dirichlet character $\mod N$ and write $\chi=\chi^{(p^k)}\chi^{(M)}$ where $N=p^kM, (p,M)=1$. The de-adelization of a central element of $\mathcal{H}(G/\!\!/K_0(p^n),\bar{\chi}_{id})$ acting on $A_k(N,\chi_{id})$ is: $q(\mathcal{Z}^n)(f)=\chi^{(M)}(p)^nf$. 
\end{lemma}
\begin{proof}
Let $\Phi_f=q^{-1}(f)$ be the idelization of $f$ in $A_k(N,\chi_{id})$.
For the action of the center we have that $\int_{G}\Phi_f(xg)\mathcal{Z}^n(g)dg=\Phi_f(gz(p^n))$ which means $q(\mathcal{Z}^n)(f)(z)=\Phi_f(g_{\infty}k_0z(p^n))j(g_{\infty},i)^{k}=\\ \chi_{id}(1_2,1_3,...,p^n_p,...)\Phi_f(g_{\infty}k_0)j(g_{\infty},i)^{k}= \chi_{id}(1_2,1_3,...,p^n_p,...)f(z)$. But if $N=\prod\limits_{i=1}^mp_i^{k_i}$ then $\chi_{id}=\prod\limits_{i=1}^m\chi_{id}^{(p_i^{k_i})}$. Writing $(1_2,1_3,...,p^n_p,...)=p^n(p^{-n}_2,p^{-n}_3,...,1_p,...)$ we get $\chi_{id}^{(p_i^{k_i})}(1_2,1_3,...,p^n_p,...)=\overline{\chi}^{(p_i^{k_i})}(p^{-n})=\chi^{(p_i^{k_i})}(p^{n})$ except at $p_i=p$ where we get $\chi_{id}^{(p^k)}(p_2,p_3,...,1,...)=\chi^{(p^k)}(1)=1$. The claim now follows at once.
\end{proof}
\begin{lemma}
    With notation as before in $\mathcal{H}(K/\!\!/K_0(p^n),\bar{\chi}_{id})$ we have $q(\mathcal{V}_{\ell}^{\chi})(f)=\sum\limits_{s\in{\mathbb{Z}_p}^*/1+p^{n-\ell}\mathbb{Z}_p}\bar{\chi}(A_{s,\ell})f|A_{s,\ell}$. In particular $$q(\mathcal{Y}^{\chi}_{r})(f)=S_{p^n,r}(f)=f+\sum_{j=r}^{n-1}\sum_{s\in{\mathbb{Z}_p}^*/1+p^{n-j}\mathbb{Z}_p}\bar{\chi}(A_{s,j})f|A_{s,j}$$ where $A_{s,j}\in SL_2(\mathbb{Z})$ is any matrix of the form $\begin{pmatrix}
    a_{s,j} & b_{s,j} \\
    p^jM & p^{n-j}-sM
\end{pmatrix}$.
\end{lemma}
\begin{proof}
    We have $\mathcal{V}^{\chi}_{\ell}(\Phi)(x)=\int_G\Phi(xg)\mathcal{Y}^{\chi}_{\ell}(g)dg=\sum\limits_{s\in\mathbb{Z}^*_p/(1+p^{n-\ell}\mathbb{Z}_p)}\Phi(xd(s)y(p^{\ell}))$. We need to compute 
$q(\mathcal{V}^{\chi}_{\ell})(f)(z)=\sum\limits_{s\in\mathbb{Z}^*_p/(1+p^{n-\ell}\mathbb{Z}_p)}\Phi_f(g_{\infty}d(s)y(p^{\ell}))j(g_{\infty},i)^{-k}$. Let $A_{s,\ell}$ be as above, then $A_{s,\ell}d(s)y(p^{\ell})\in K_0(N)$ and it is easy to see that $\chi_{id}(A_{s,\ell}d(s)y(p^{\ell})=\bar{\chi}(A_{s,\ell})$. We then obtain $q(\mathcal{V}^{\chi}_{\ell})(f)(z)=\\ \sum\limits_{s\in\mathbb{Z}^*_p/(1+p^{n-\ell}\mathbb{Z}_p)}\chi_{id}(A_{s,\ell})\Phi_f(A_{s,\ell}g_{\infty})j(g_{\infty},i)^{-k}=\sum\limits_{s\in\mathbb{Z}^*_p/(1+p^{n-\ell}\mathbb{Z}_p)}\bar{\chi}(A_{s,\ell})f|_{A_{s,\ell}}(z)$. The formula for $\mathcal{Y}^{\chi}_{r}$ now follows immediately from the definition. 
\end{proof}
\section{Square-free Level Forms} \label{Sqfree}
Let $p|N,p^2\nmid N$ and $M=N/p$, then we can write $\chi=\chi^{(p)}\chi^{(M)}$ where $\chi^{(p)}$ is the Dirichlet character of $(\mathbb{Z}/p\mathbb{Z})^{\times}$ through which $\chi$ factors and similarly for $\chi^{(M)}$. By definition in order to have oldforms of level $M$, the space $\mathcal{S}_{k}(\Gamma_0(M),\chi)$ has to be nontrivial which can only hold if and only if $\chi$ is also a character of $\Gamma_0(M)$ (see \cite[Corollary 1]{Li}). This implies that $\mathcal{S}_{k}^{old}(\Gamma_0(N),\chi)$ is non-trivial if and only if $\chi$ is not primitive $\mod N$. When $N$ is square-free this is equivalent to $\chi^{(p)}\equiv 1$ for some prime divisor $p$ of $N$. From this it follows that the corresponding Hecke algebra is the Iwahori algebra $\mathcal{H}(G/\!\!/K_0(p))$. We now use the relation $\mathcal{U}_0=\mathcal{Z}^{-1}*\mathcal{T}_1*\mathcal{U}_1$. The results we obtained in Section 2 enable us to  define the Hecke operator $Q_p=q(\mathcal{U}_0)=q(\mathcal{Z}^{-1}\mathcal{T}_1\mathcal{U}_1)=\bar{\chi}^{(M)}(p)\tilde{U}_pW_p$ which from Theorem \ref{struct} satisfies a quadratic relation $(Q_p+1)(Q_p-p)=0$. We then have the following:

\begin{lemma} \label{Q_plem}
If $f\in \mathcal{S}_{k}^{new}(\Gamma_0(N),\chi)$ then $Q_p(f)(z)=-f(z)$.
\end{lemma}
\begin{proof}
We first note that $\mathcal{S}_{k}(\Gamma_0(N),\chi)^{new}$ has a basis consisting of eigenfunctions and thus it is sufficient to prove the result for these basis elements. Let $f\in \mathcal{S}_{k}^{new}(\Gamma_0(N),\chi)$ be an eigenform then, by \cite[Theorem 3]{Li} we obtain that $W_p(f)(z)=-a(p)p^{1-\frac{k}{2}}f(z)$. This implies that \\ $Q_p(f)(z)=-p^{1-\frac{k}{2}}\bar{\chi}^{(M)}(p)a(p)p^{1-\frac{k}{2}}\tilde{U}_p(f)(z)=-p^{2-k}a(p)^2\bar{\chi}^{(M)}(p)f(z)$. Using \cite[Theorem 3]{Li} again yields $a(q)^2=\chi_{M}(p)p^{k-2}$ and thus $Q_p(f)(z)=-f(z)$. 
\end{proof}
From now on we always assume that $\chi=\chi^{(p)}\chi^{(M)}$ with $\chi^{(p)}\equiv 1$ so that $\mathcal{S}_{k}(\Gamma_0(M),\chi)\ne\{0\}$, otherwise there are no oldforms coming from level $M$ forms. 
\begin{lemma} \label{lem2}
Let $f\in \mathcal{S}_{k}(\Gamma_0(M),\chi)$ then $Q_p(f)(z)=pf(z)$. 
\end{lemma}
\begin{proof}
We have that $W_p(f)(\frac{z+s}{p})=p^{\frac{k}{2}}f|\begin{pmatrix}
    \beta & 1 \\
    M\gamma & p
\end{pmatrix}(z+s)=\\ p^{\frac{k}{2}}f|\begin{pmatrix}
    \beta & 1 \\
    M\gamma & p
\end{pmatrix}(z)$. And since $\begin{pmatrix}
    \beta & 1 \\
    M\gamma & p
\end{pmatrix}\in\Gamma_0(M)$ we get that \\ $p^{\frac{k}{2}}f|\begin{pmatrix}
    \beta & 1 \\
    M\gamma & p
\end{pmatrix}(z)=p^{\frac{k}{2}}\chi^{(M)}(p)f(z)$. Hence: $Q_p(f)(z)= \\ \bar{\chi}^{(M)}(p)p^{-\frac{k}{2}}\sum\limits_{s=0}^{p-1}W_p(f)(\frac{z+s}{p})=\bar{\chi}^{(M)}(p)p^{-\frac{k}{2}}\sum\limits_{s=0}^{p-1}p^{\frac{k}{2}}\chi^{(M)}(p)f(z)=pf(z)$.
\end{proof}
Consider now the operator $Q_p'=W_pQ_pW_p^{-1}$ formed by conjugating $Q_p$ with $W_p$. Then $(Q_p'+1)(Q_p'-p)=0$ and $Q_p'=W_p\tilde{U}_p$. Since $W_p$ is a surjective endomorphism of $\mathcal{S}_{k}^{new}(\Gamma_0(N),\chi)$ when $\chi^{(p)}\equiv 1$ we immediately have the following lemma:
\begin{lemma} \label{Qp'lem}
If $f\in \mathcal{S}_{k}^{new}(\Gamma_0(N),\chi)$ then $Q_p'(f)(z)=-f(z)$. 
\end{lemma}
For the space of oldforms we have the following decomposition: \\ $\mathcal{S}_{k}^{old}(\Gamma_0(N),\chi)=\sum\limits_{\text{primes }p|N:\chi^{p}\equiv 1}\mathcal{S}_{k}(\Gamma_0(N/p),\chi^{(N/p)})\oplus V(p)\mathcal{S}_{k}(\Gamma_0(N/p),\chi^{(N/p)})$, \\ where $V(p)f(z)=f(pz)$. Now we will examine how $W_p$ acts on each of these components.
\begin{lemma} \label{lem1}
    Let $f\in \mathcal{S}_{k}^{old}(\Gamma_0(M),\chi)$, then $W_p(f)(z)=p^{\frac{k}{2}}\chi^{(M)}(p)f(pz)$. If $g(z)=f(pz)$ then $W_p(g)(z)=p^{-\frac{k}{2}}f(z)$. Consequently $W_p$ maps $\mathcal{S}_{k}(\Gamma_0(M),\chi)$ to $V(p)\mathcal{S}_{k}(\Gamma_0(M),\chi)$ isomorphically. 
\end{lemma}
\begin{proof}
	Let $f\in \mathcal{S}_{k}^{old}(\Gamma_0(M),\chi)$, then we have that\\ $W_p(f)(z)=p^{\frac{k}{2}}(M\gamma(pz)+p)^{-k}f(\frac{\beta(pz)+1}{M\gamma(pz)+p})=p^{\frac{k}{2}}f|\begin{pmatrix}
    \beta & 1 \\
    M\gamma & p
\end{pmatrix}(pz)=\\ p^{\frac{k}{2}}\chi^{(M)}(p)f(pz)$, since $\begin{pmatrix}
    \beta & 1 \\
    M\gamma & p
\end{pmatrix}\in\Gamma_0(M)$.
\par Conversely if $g(z)=f(pz), f\in \mathcal{S}_{k}^{old}(\Gamma_0(M),\chi)$ then similarly \\ $W_p(g)(z)=p^{\frac{k}{2}}(pM\gamma z+p)^{-k}f(p\frac{\beta(pz)+1}{M\gamma(pz)+p})=p^{-\frac{k}{2}}(M\gamma z+1)^{-k}f(\frac{\beta(pz)+1}{M\gamma z+1})=p^{-\frac{k}{2}}f|\begin{pmatrix}
    p\beta & 1 \\
    M\gamma & 1
\end{pmatrix}=p^{-\frac{k}{2}}f(z)$, since $\begin{pmatrix}
    p\beta & 1 \\
    M\gamma & 1
\end{pmatrix}\in\Gamma_0(M)$. \par
Thus $W_p$ maps $\mathcal{S}_{k}(\Gamma_0(M),\chi)$ to $V(p)\mathcal{S}_{k}(\Gamma_0(M),\chi)$ with inverse $\bar{\chi}^{(M)}(p)W_p$ which shows that this is indeed an isomorphism.
\end{proof}
\begin{cor} \label{corol}
    Lemma \ref{lem1} implies that $V(p)\mathcal{S}_{k}(\Gamma_0(M),\chi)$ is contained in the $p$-eigenspace of $Q_p'$.
\end{cor}
Next we show that the operators we constructed above are Hermitian with respect to the Petersson inner product $\langle.,.\rangle$. Further, let us denote the p-adic inner product as $\langle.,.\rangle_p$.
\begin{prop} \label{herm}
    The Hecke operator $Q_p$ is Hermitian and thus $Q_p'$ is Hermitian. 
\end{prop}
\begin{proof}
By Gelbart (\cite[p.88]{Gel}) $\langle\Phi_f,\Phi_g\rangle_p=\langle f,g\rangle$ thus $$\langle\mathcal{U}_0\Phi_f,\Phi_g\rangle_p=\langle\Phi_{Q_pf},\Phi_g\rangle_p=\langle Q_pf,g\rangle.$$ This implies that: $$\langle Q_pf,g\rangle=\langle\mathcal{U}_0\Phi_f,\Phi_g\rangle_p=\langle\mathcal{Z}^{-1}\mathcal{T}_1\mathcal{U}_1\Phi_f,\Phi_g\rangle_p.$$ And since  $\langle\mathcal{T}_1\mathcal{U}_1\Phi_f,\Phi_g\rangle_p=\langle \mathcal{U}_1\Phi_f,\mathcal{T}_{-1}\Phi_g\rangle_p= \langle \mathcal{U}_1\Phi_f,\mathcal{Z}^{-1}\mathcal{U}_1\mathcal{T}_1\mathcal{U}_1^{-1}\Phi_g\rangle_p=\\ \langle \Phi_f,\mathcal{Z}^{-1}\mathcal{T}_1\mathcal{U}_1^{-1}\Phi_g\rangle_p\overset{\mathcal{U}_1^2\mathcal{Z}^{-1}=I}{=} \langle \Phi_f,\mathcal{Z}^{-1}\mathcal{T}_1\mathcal{U}_1^{-1}\mathcal{U}_1^2\mathcal{Z}^{-1}\Phi_g\rangle_p=\langle \Phi_f,\mathcal{Z}^{-2}\mathcal{T}_1\mathcal{U}_1\Phi_g\rangle_p$, we obtain the relation: $$\langle Q_pf,g\rangle=\langle \mathcal{Z}^{-1}\Phi_f,\mathcal{Z}^{-2}\mathcal{T}_1\mathcal{U}_1\Phi_g\rangle_p=\langle \Phi_f,\mathcal{U}_0\Phi_g\rangle_p=\langle f,Q_pg\rangle$$

\par
For $Q_p'$ then: $\langle W_pQ_pW_p^{-1}f,g\rangle=\langle Q_pW_p^{-1}f,W_p^{-1}g\rangle=\langle W_p^{-1}f,Q_pW_p^{-1}g\rangle=\langle f,W_pQ_pW_p^{-1}g\rangle$.
\end{proof}
We are now able to prove the main theorem for square-free $N$.
\begin{theorem} \label{sqfr}
    Consider $\mathcal{S}_{k}(\Gamma_0(N),\chi)$ where $N=\prod\limits_{i=1}^mp_i$ and write $\chi=\prod\limits_{i=1}^m\chi^{(p_i)}$. Then $\mathcal{S}_{k}^{new}(\Gamma_0(N),\chi)$ is the intersection of the $-1$ eigenspaces of the Hecke operators $Q_{p_i},Q_{p_i}'$ for every $i$ with $\chi^{(p_i)}\equiv 1$
\end{theorem}
\begin{proof}
    If $f\in \mathcal{S}_{k}^{new}(\Gamma_0(N),\chi)$ then Lemma
    \ref{Q_plem} and Lemma \ref{Qp'lem} state that $Q_p(f)=Q_p'(f)=-f$.
    Let $f\in \mathcal{S}_{k}(\Gamma_0(N),\chi)$ be such that $Q_p(f)=Q_p'(f)=-f$  for every prime $p|N$. Since $Q_p,Q_p'$ are pseudo-Hermitian by Proposition \ref{herm}, $f$ is orthogonal to their resepctive $p$-eigenspaces. The $p-$eigenspace of $Q_p$ however contains $\mathcal{S}_{k}(\Gamma_0(M),\chi)$ by Lemma \ref{lem2} and similarly by Corollary \ref{corol} we have that $V(p)\mathcal{S}_{k}(\Gamma_0(M),\chi)$ is in the $p-$eigenspace of $Q_p'$. Thus $f$ is orthogonal to $\mathcal{S}_{k}(\Gamma_0(M),\chi)\oplus V(p)\mathcal{S}_{k}(\Gamma_0(M),\chi)$ for every $p|N$, completing the proof.
\end{proof}

\section{General Level Forms} \label{GenLev}
Let $N=p^nM: (M,p)=1$ and $\chi=\prod_{i=1}^m\chi^{({p_i}^{a_i})}=\chi^{(p^n)}\chi^{(M)}$. Consider the classical Hecke operator $S_{p^n,r}(f)=f+\sum_{j=r}^{n-1}\sum_{s\in{\mathbb{Z}_p}^*/1+p^{n-j}\mathbb{Z}_p}\bar{\chi}(A_{s,j})f|A_{s,j}$ where $A_{s,j}=\begin{pmatrix}
    a_{s,j} & b_{s,j} \\
    p^jM & p^{n-j}-sM
\end{pmatrix}$ with $a_{s,j}, b_{s,j}$ chosen such that $det(A_{s,j})=1$. From \cite[Lemma 5.8]{BarSom} then we use the following lemma:
\begin{lemma}
    For $1\leq r\leq n$ a set of right coset representatives of $\Gamma_0(N)$ in $\Gamma_0(p^rM)$ consists of $I$ and $A_{s,j}$ with $s\in{\mathbb{Z}_p}^*/1+p^{n-j}\mathbb{Z}_p$ and $r\leq j\leq n-1$.
\end{lemma}
We will use the lemma below (see \cite[p. 134]{Miy}) to show that assuming $\chi^{({p}^{n})}$ is of conductor at most $p^{c}$ we have that $S_{p^n,r}:\mathcal{S}_{k}(\Gamma_0(N),\chi)\to \mathcal{S}_{k}(\Gamma_0(p^rM),\chi)$. 
\begin{lemma} [\textbf{Miyake} p.134]
Let $\Gamma_1\subseteq\Gamma_2$ and suppose $f\in \mathcal{S}_k(\Gamma_1,\chi)$ where $\chi:\Gamma_2\to\mathbb{C}^{\times}$ is a character of $\Gamma_2$. If $\{A_1,...,A_n\}$ is a set of coset representatives of $\Gamma_1$ in $\Gamma_2$ then the operator $S(f)(z)=\sum_{i=1}^n\bar{\chi}(A_i)f|A_i(z)$ is an operator $S:\mathcal{S}_{k}(\Gamma_1,\chi)\to \mathcal{S}_{k}(\Gamma_2,\chi)$.
\end{lemma}
\begin{proof}
    For every $\gamma'\in\Gamma_2$ and $i\in\{1,...,n\}$, there exists a $\gamma_i\in\Gamma_1:A_i\gamma'=\gamma_iA_j$ and distinct $i$ give rise to distinct $j$. Then $S(f)|\gamma'=\sum\limits_{i=1}^n\bar{\chi}(A_i)f|A_i\gamma'=\sum\limits_{i=1}^n\bar{\chi}(A_i)f|\gamma_iA_j=\sum\limits_{i=1}^n\bar{\chi}(A_i)\chi(\gamma_i)f|A_j=\sum\limits_{i=1}^n\bar{\chi}(A_i)\chi(A_i\gamma'A_j^{-1})f|A_j\overset{(**)}{=}\\ \sum\limits_{j=1}^n\bar{\chi}(A_j)\chi(\gamma')f|A_j=\bar{\chi}(\gamma')S(f)$. Equality $(**)$ follows because each $j$ appears exactly once for each $i$ and $\chi(A_i\gamma'A_j^{-1})=\chi(A_i)\chi(\gamma')\bar{\chi}(A_j)$.
\end{proof}
\begin{cor} \label{cor2}
    If $\chi^{({p}^{n})}$ is of conductor $p^c$ and $r\geq c$, then $S_{p^n,r}:\mathcal{S}_{k}(\Gamma_0(N),\chi)\to \mathcal{S}_{k}(\Gamma_0(p^rM),\chi)$.
\end{cor}
Below we will keep this assumption on the conductor of $\chi$ and every operator we define is obviously for $r\geq c$ unless otherwise stated. Corollary \ref{cor2} immediately leads to the following result:
\begin{prop} \label{iszer}
    Let $f\in \mathcal{S}_{k}^{new}(\Gamma_0(N),\chi)$, then $S_{p^n,r}(f)=0$ where $r\geq c$.
\end{prop}
\begin{proof}
    It suffices to prove the result for every form $f$ in an eigenbasis of the space $\mathcal{S}_{k}^{new}(\Gamma_0(N),\chi)$.
    Since $S_{p^n,r}=q(\mathcal{Y}_r)$ and $\mathcal{Y}_r$ commutes with $\mathcal{T}_q$ for every prime $q\nmid N$ we have that $S_{p^n,r}$ also commutes with $T_q$ classicaly. Then $T_q(S_{p^n,r}(f))=S_{p^n,r}(T_q(f))\implies T_q(S_{p^n,r}(f))=a(q)S_{p^n,r}(f)$ and thus $S_{p^n,r}(f)\sim f$. But by Corollary \ref{cor2}: $S_{p^n,r}(f)\in \mathcal{S}_{k}(\Gamma_0(p^rM))$ and is equivalent to a newform. By \cite[Theorem 2]{Li} it follows that $S_{p^n,r}(f)=0$.
\end{proof}
We now investigate the action of $S_{p^n,r}$ on the lower-level forms.  
\begin{lemma}
    Let $f\in \mathcal{S}_{k}(\Gamma_0(p^rM),\chi)$, then $S_{p^n,r}(f)=p^{n-r}f$. Equivalently: $\mathcal{S}_{k}(\Gamma_0(p^rM),\chi)$ lies in the $p^{n-r}$ eigenspace of $S_{p^n,r}$. 
\end{lemma}
\begin{proof}
    Notice that $f\in \mathcal{S}_{k}(\Gamma_0(p^rM),\chi)\implies f|A_{s,j}=\chi(A_{s,j})f$. We then have $S_{p^n,r}(f)=f+\sum\limits_{j=r}^{n-1}\sum\limits_{s\in{\mathbb{Z}_p}^*/1+p^{n-j}\mathbb{Z}_p}\bar{\chi}(A_{s,j})\chi(A_{s,j})f= f+\sum\limits_{j=r}^{n-1}\sum\limits_{s\in{\mathbb{Z}_p}^*/1+p^{n-j}\mathbb{Z}_p}f=p^{n-r}f$.
\end{proof}
We also consider the operator $S_{p^n,r}'=W_{p^n}S_{p^n,r}W_{p^n}^{-1}$. Notice that from operator similarity we obtain that $S_{p^n,r}'(S_{p^n,r}'-p^{n-r})=0$. One important thing to note is that $S_{p^n,r}$ here has domain $\mathcal{S}_{k}(\Gamma_0(N),\bar{\chi}^{(p^n)}\chi^{(M)})$ since it takes arguments of the form $W_{p^n}^{-1}(f)$. 
\begin{lemma} \label{iszer2}
    Let $f\in \mathcal{S}_{k}^{new}(\Gamma_0(N),\chi)$, then $S_{p^n,r}'(f)=0$.
\end{lemma}
\begin{proof}
    By \cite[p. 224]{Li-Atkin}, $W_{p^n}$ and thus also $W_{p^n}^{-1}=\chi^{(p^n)}(-1)\bar{\chi}^{(M)}(p^n)W_{p^n}$ preserve newforms which implies  $W_{p^n}:\mathcal{S}_{k}^{new}(\Gamma_0(N),\bar{\chi}^{(p^n)}\chi^{(M)})\to \mathcal{S}_{k}^{new}(\Gamma_0(N),\chi)$ is onto. Writing any $f\in \mathcal{S}_{k}^{new}(\Gamma_0(N),\chi)$ as $f(z)=W_p(g)(z)$ where \\$g\in \mathcal{S}_{k}^{new}(\Gamma_0(N),\bar{\chi}^{(p^n)}\chi^{(M)})$ immediately implies the claim.
\end{proof}
\begin{lemma} \label{W_pn}
    For $c\leq r\leq n$ we have that $W_{p^n}$ maps $\mathcal{S}_{k}(\Gamma_0(p^rM),\bar{\chi}^{(p^n)}\chi^{(M)})$ onto $V(p^{n-r})\mathcal{S}_{k}(\Gamma_0(p^rM),\chi^{(p^n)}\chi^{(M)})$. 
\end{lemma}
\begin{proof}
	For any $f\in \mathcal{S}_{k}(\Gamma_0(p^rM),\bar{\chi}^{(p^n)}\chi^{(M)})$ we obtain that $W_{p^n}(f)=\\p^{\frac{nk}{2}}(p^rM\gamma(p^{n-r}z)+p^n)^{-k}f(\frac{p^r\beta(p^{n-r}z)+1}{p^rM\gamma(p^{n-r}z)+p^n})=p^{\frac{(n-r)k}{2}}f|W_{p^r}(p^{n-r}z)$. It remains to show that $\forall g\in \mathcal{S}_{k}(\Gamma_0(p^rM),\chi^{(p^n)}\chi^{(M)})$, there exists \\$g'\in \mathcal{S}_{k}(\Gamma_0(p^rM),\bar{\chi}^{(p^n)}\chi^{(M)})$ such that $g(z)=W_{p^r}(g)(z)$. We thus set $g'(z)=\chi^{(p^n)}(-1)\bar{\chi}^{(M)}(p^r)g|W_{p^r}(z)$ and the claim $g(z)=W_{p^r}(g)(z)$ follows from the identity $f|W_{p^r}|W_{p^r}=\chi^{(p^n)}(-1)\chi^{(M)}(p^r)f$. Applying the above calculation then: $W_{p^n}(p^{\frac{(r-n)k}{2}}g')(z)=g'|W_{p^r}(p^{n-r}z)=g(p^{n-r}z)$ for any choice of $g\in \mathcal{S}_{k}(\Gamma_0(p^rM),\chi^{(p^n)}\chi^{(M)})$ proving our map is onto.
\end{proof}

Since the oldspace can be decomposed as: \[\mathcal{S}^{old}_{k}(\Gamma_0(N),\chi)=\sum\limits_{p|N: \chi^{(p^n)}\text{non-primitive}}\mathcal{S}_{k}(\Gamma_0(N/p),\chi^{(N/p)})\oplus V(p)\mathcal{S}_{k}(N/p),\chi^{(N/p)})\] we only need to consider $S_{p^n,n-1}(f)=f+\sum\limits_{s\in{\mathbb{Z}_p}^*/1+p\mathbb{Z}_p}\bar{\chi}(A_s)f|A_s$, where $A_s\in SL_2(\mathbb{Z})$ is any matrix of the form $\begin{pmatrix}
    a_{s} & b_{s} \\
    p^{n-1}M & p-sM
\end{pmatrix}$.
\begin{cor}
    Lemma \ref{W_pn} implies that $V(p)\mathcal{S}_{k}(\Gamma_0(p^{n-1}M),\chi)$ is contained in the $p$-eigenspace of $S_{p^n,n-1}'$.
\end{cor}
We now establish again the Hermitian property for $S_{p^n,n-1}$ and the pseudo-Hermitian property for $S_{p^n,n-1}'$.
\begin{prop}
    The operator $S_{p^n,n-1}$ is Hermitian and thus $S_{p^n,n-1}'$ is Hermitian.
\end{prop}
\begin{proof}
    Consider $S_{p^n,n-1}(f)=f+\sum_{s\in{\mathbb{Z}_p}^*/1+p\mathbb{Z}_p}\bar{\chi}(A_s)f|A_s$. From the proof of \cite[Proposition 5.14]{BarSom} we know that $\forall s\in{\mathbb{Z}_p}^*/1+p\mathbb{Z}_p,\exists !t\in{\mathbb{Z}_p}^*/1+p\mathbb{Z}_p:A_tA_s=\gamma_s\in\Gamma_0(N)$. Then $\langle S_{p^n,n-1}(f),g\rangle=\sum\limits_{s\in{\mathbb{Z}_p}^*/1+p\mathbb{Z}_p}\bar{\chi}(A_s)\langle f|A_s,g\rangle=\sum\limits_{s\in{\mathbb{Z}_p}^*/1+p\mathbb{Z}_p}\langle f,\chi(A_s)g|A_s^{-1}\rangle=\sum\limits_{s\in{\mathbb{Z}_p}^*/1+p\mathbb{Z}_p}\langle f,\chi(A_s)g|\gamma_s^{-1}A_t\rangle\overset{\chi(\gamma_s^{-1})=\bar{\chi}(A_s)\bar{\chi}(A_t)}{=}\sum\limits_{s\in{\mathbb{Z}_p}^*/1+p\mathbb{Z}_p}\langle f,\bar{\chi}(A_t)g|A_t\rangle= \langle f,S_{p^n,n-1}(g)\rangle$. \par
    The proof for $S_{p^n,n-1}$ follows in the same way as $Q_p'$. 
\end{proof}
We are finally in position to prove \textbf{Theorem} \ref{mainth}:
\begin{proof}
    Let $p|N$ be prime. If $p^2\nmid N$ we repeat the proof of Theorem \ref{sqfr}. Otherwise from Proposition \ref{iszer} and Lemma \ref{iszer2} we know that $f\in \mathcal{S}_{k}^{new}(\Gamma_0(N),\chi) \implies S_{p^n,n-1}(f)=S'_{p^n,n-1}(f)=0$. In the other direction, suppose $S_{p^n,n-1}(f)=S'_{p^n,n-1}(f)=0$ for every prime such that $p^2|N$. We know by the Hermitian property of $S_{p^n,n-1}$ that $f$ is orthogonal to it's $p$-eigenspace and similarly for the $p$-eigenspace of $S'_{p^n,n-1}$. But since the $p$-eigenspace of $S_{p^n,n-1}$ contains $\mathcal{S}_{k}(\Gamma_0(p^{n-1}M),\chi)$ and similarly the $p$-eigenspace of $S_{p^n,n-1}'$ contains $V(p)\mathcal{S}_{k}(\Gamma_0(p^{n-1}M),\chi)$, we have that $f$ is orthogonal to $\mathcal{S}_{k}(\Gamma_0(p^{n-1}M),\chi)\oplus V(p)\mathcal{S}_{k}(\Gamma_0(p^{n-1}M),\chi)$ for every divisor of $N$, including the square-free primes from Theorem \ref{sqfr}. This concludes the proof.
\end{proof}


\begin{thebibliography}{99}

\bibitem{Atkin-Lehner}
A.O.L Atkin, J. Lehner, Hecke Operators on this $\Gamma_0(m)$, Math. Ann. 185 (1970), p.134-160

\bibitem{Li-Atkin}
A.O.L Atkin, Wen-Ching Winnie Li, Twists of Newforms and Pseudo-Eigenvalues of {W}-Operators*, Inv. Math. 48 (1978), p.221-243

\bibitem{BarSom1}
E.M. Baruch, S. Purkait, Newforms of Half-Integral weight: The Minus Space counterpart, Canadian J. of Math. 72, 2 (2020), p. 326-372

\bibitem{BarSom}
E.M. Baruch, S. Purkait, Hecke algebras, new vectors and newforms on ${\Gamma}_0(m)$, Mathematische Zeitschrift 287 1-2 (2017), p.705-733

\bibitem{Cass}
W. Casselman, The Restriction of a representation of ${GL}_2(k)$ to ${GL}_2(\mathcal{O})$, Math. Ann. 206 (1973) p.311-318

\bibitem{CassLehn}
W. Casselman, On some results of Atkin and Lehner, Math. Ann. 201 (1973), p. 301–314 

\bibitem{Gel}
S. Gelbart, Automorphic Forms on Adele Groups, 1975, Princeton University Press and University of Tokyo Press

\bibitem{Gold}
D. Goldfeld, J. Hundley, Automorphic Representations and ${L}$-functions for the {G}eneral {L}inear {G}roup, vol. 1, 2011, Cambridge University Press 

\bibitem{How}
R. Howe, Affine-like Hecke algebras and $p$-adic representation theory in Iwahori-Hecke Algebras and their Representation Theory, Lecture Notes in Mathematics 1804 (2002), p. 27-69

\bibitem{LokSav}
H. Y. Loke and G. Savin, Representations of the two-fold central extension of $SL_2(\mathbb{Q}_2)$,
Pacific J. Math. 247 (2010), 435–454

\bibitem{Miy}
T. Miyake, Modular Forms, 1989, Springer Berlin Heidelberg New York

\bibitem{Li}
Wen-Ching Winnie Li, Newforms and Functional Equations, Math. Ann. 212 (1975) p. 285-315

\end{thebibliography}
\end{document}